\documentclass[11pt]{amsart}

\usepackage {amsmath, amssymb,   epsfig, a4wide, enumerate, psfrag}
\usepackage[utf8]{inputenc}
\usepackage[english]{babel}
\usepackage{pst-node}

\usepackage{tikz-cd} 

\numberwithin{equation}{section}

\date{\today}

\keywords{Post-critically finite endomorphism, Julia set, laminar current, Fatou disk.}
\author{Zhuchao Ji}

\title{STRUCTURE OF JULIA SETS FOR POST-CRITICALLY FINITE
ENDOMORPHISMS ON $\mathbb{P}^2$}

\address{Sorbonne Université, Laboratoire de Probabilités, Statistique et Modélisation (LPSM, UMR 8001),   4 place Jussieu, 75252 Paris Cedex 05, France}
\email{zhuchaoji.math@gmail.com}

\subjclass[2010]{37F15, 37F50, 32U40.}

\newtheorem{theorem}{Theorem}[section]
\newtheorem{definition}[theorem]{Definition}
\newtheorem{proposition}[theorem]{Proposition}
\newtheorem{corollary}[theorem]{Corollary}
\newtheorem{lemma}[theorem]{Lemma}

\newtheorem*{theorem 4.4}{Theorem 4.3}
\newtheorem*{theorem 6.1}{Theorem 6.1}
\newtheorem*{proposition 5.5}{Proposition 5.5}

\newtheorem*{theorem*}{Theorem}

\begin{document}

	\maketitle

	\begin{abstract}
Let $f$ be a post-critically finite endomorphism (PCF map for short) on $\mathbb{P}^2$, let $J_1$ denote the Julia set and let $J_2$ denote the support
of the measure of maximal entropy. In this paper we show that: 1. $J_1\setminus J_2$ is contained in the union of the (finitely many) basins of  critical component cycles and stable manifolds of  sporadic super-saddle cycles. 2. For every $x\in J_2$ which is not contained in the stable manifold of a sporadic super-saddle cycle, there is no Fatou disk  containing $x$. Here sporadic means that the super-saddle cycle is not contained in a critical component cycle.  Under the additional assumption that all branches
of $PC(f)$ are smooth and intersect transversally, we show that there is no sporadic super-saddle cycle. Thus in this case  $J_1\setminus J_2$ is contained in  the union of the basins of  critical component cycles, and for every $x\in J_2$  there is no Fatou disk  containing $x$. 
\par As   consequences of our result: 1.We answer some questions of Fornaess-Sibony about the non-wandering set for PCF maps on $\mathbb{P}^2$ with no sporadic super-saddle cycles. 2. We give a new proof of de Th\'elin's laminarity of the Green current in $J_1\setminus J_2$ for PCF maps on $\mathbb{P}^2$. 3. We show that  for PCF maps on $\mathbb{P}^2$ an invariant compact set is expanding if and only if it does not contain critical points, and we obtain characterizations of PCF maps on $\mathbb{P}^2$ which are expanding on $J_2$ or satisfy Axiom A.
	\end{abstract}
\section{Introduction}
\subsection{Background}
Let $f : \mathbb{P}^2\to \mathbb{P}^2$ be a holomorphic endomorphism of degree $\geq 2$, where
$\mathbb{P}^2$ is the complex projective plane. The  {\em first Julia set} $J_1$ is defined as the locus where the
iterates $(f^n)_{n\geq 0}$ do not locally form a normal family, i.e. the complement of the {\em Fatou set}.
Let $T$ be the {\em dynamical Green current} of $f$, defined by $T = \lim_{n \to +\infty}d^{-n}(f^n)^*\omega_{FS}$, where $\omega_{FS}$ is
the Fubini-Study (1,1) form on $\mathbb{P}^2$. The Julia set $J_1$ coincides with $\text{Supp}\;(T)$, and the self intersection measure $\mu = T \wedge T$ is the unique measure of maximal entropy of $f$. See Dinh-Sibony \cite{dinh2010dynamics} for
background on holomorphic dynamics on projective spaces.

\medskip
\par We define the {\em second Julia set} to be $J_{2}=$ Supp $\mu$. From the definitions we know that $J_{2} \subset J_{1}$. By Briend-Duval \cite{briend1999exposants}, $J_2$ is contained in the closure of the set of repelling periodic points. However contrary to the situation in dimension one there may exists repelling periodic point outside $J_2$. A major problem in holomorphic dynamics is to investigate the structure of $J_{1} \setminus J_{2}$. A promising picture is that $J_{1} \setminus J_{2}$ is foliated (in some appropriate sense) by holomorphic disks $D$ along which $\left(\left.f^{n}\right|_{D}\right)_{n \geq 0}$ is a normal family. Such disks are called {\em Fatou disks}. The dynamical Green current $T$ is called {\em laminar} in some open set $\Omega$ if it expresses as an integral of integration currents over a measurable family of compatible holomorphic disks (which means that these disks have no isolated intersections) in $\Omega.$ These disks are automatically Fatou disks. Let $\sigma_{T}=T \wedge \omega$ be the trace measure of $T,$ which is a natural reference measure on $J_{1}$. If $T$ is laminar in $\Omega$, then for $\sigma_{T}$
a.e. $x \in \Omega$ there exists a germ of holomorphic Fatou disk $D$ containing $x$. De Thélin proved in \cite{de2004laminarite} and \cite{de2005phenomene} that $T$ is laminar in $J_1\setminus J_2$ for {\em post-critically finite endomorphisms} on $\mathbb{P}^{2}.$ We note that by the works of Dujardin in general $T$ is not necessarily laminar in $J_1\setminus J_2$  \cite{dujardin2016non}, but a related weaker result holds \cite{dujardin2012fatou}.
\medskip

A holomorphic endomorphism $f$ on $\mathbb{P}^{k}, k \geq 1$ is called post-critically finite (PCF for short)
if the post-critical set
\begin{equation*}
PC(f):=\bigcup_{n \geq 1} f^{n}(C(f))
\end{equation*}
is an algebraic subset of $\mathbb{P}^{k},$ where
\begin{equation*}
C(f):=\left\{x \in \mathbb{P}^{k}: D f(x) \text { is not invertible }\right\}
\end{equation*}
is the critical set.
\medskip
\par In dimension 1 , this coincides with the usual definition of PCF maps on the Riemann sphere
$\mathbb{P}^{1}$. PCF maps play an important role in one-dimensional complex dynamics, mainly because the remarkable topological classification theorem of Thurston \cite{douady1993proof}. PCF maps are still of interest in higher dimension, and their dynamics have been investigated by many authors. Here are some  results which are useful for our purpose:
\medskip
\par 1. The Fatou set of a PCF maps on $\mathbb{P}^{2}$ is the union of the basins of super-attracting cycles (Fornaess-Sibony \cite{Fornaess1994complex}, Ueda \cite{veda1998critical} and Rong \cite{rong2008fatou}).
\medskip
\par 2.  $J_{k}=\mathbb{P}^{k}$ for strictly PCF maps on $\mathbb{P}^{k}$ ($k=2$ by Jonsson \cite{jonsson1998some}, general $k$ by Ueda \cite{uedacritically}), our main result can be seen as a generalization of Jonsson's result. 
\medskip
\par 3. The eigenvalues of periodic points of PCF maps on $\mathbb{P}^{2}$ are either 0 or larger than 1 (Le \cite{le2019fixed}).
\medskip
\par The dynamics of PCF maps on $\mathbb{P}^{k}$, $k \geq 2$ were further studied by Ueda \cite{uedacritically} and Astorg \cite{astorg2018dynamics}. Moreover, interesting examples of PCF maps were constructed by Crass \cite{crass2005family}, Fornaess-Sibony \cite{Fornaess1992critically} and Koch \cite{koch2013teichmuller}.
\medskip
\subsection{Basins of critical component cycles.} In this paper we investigate the dynamics on
the Julia sets $J_1$ and $J_2$ for PCF maps on $\mathbb{P}^{2}$. Let $f$ be a PCF map on $\mathbb{P}^{2}$. Recall that the critical set $C(f)$ is an algebraic curve. We call an irreducible component $\Lambda$ of $C(f)$ periodic if there exists an integer $n \geq 1$ such that $f^{n}(\Lambda)=\Lambda.$ Such an irreducible component is called a {\em periodic critical component}. There are finitely many periodic critical components. The set $\left\{\Lambda, f(\Lambda), \ldots, f^{n-1}(\Lambda)\right\}$ is called a {\em critical component cycle}. Similarly, a critical point $x$ satisfying $f^{n}(x)=x$ for some $n \geq 1$ is called a {\em periodic critical point}. The set $\left\{x, f(x), \ldots, f^{n-1}(x)\right\}$ is called a {\em critical point cycle}. Since
$f^{n}$ and $f$ have the same Julia sets, to investigate the structure of $J_{1} \setminus J_{2}$ we may assume that all periodic critical components are invariant.
\medskip
\par An important observation is that in $\mathbb{P}^{2}$, any invariant critical component $\Lambda$ is an {\em attracting set}. By definition an attracting set $\Lambda$ in $\mathbb{P}^{2}$ is an invariant compact subset such that there is a neighborhood $U$ of $\Lambda$ satisfying $f(U) \subset \subset U,$ and $\Lambda=\bigcap_{n\geq 1} f^{n}(U).$ The open set $U$ is called
a {\em trapping region} of $\Lambda .$ The attracting basin $\mathcal{B}(\Lambda)$ of an attracting set $\Lambda$ is by definition the
set $\bigcup_{n \geq 0} f^{-n}(U),$ where $U$ is a trapping region of $\Lambda .$ Equivalently $\mathcal{B}(\Lambda)$ is the set of points attracted by $\Lambda,$ i.e.
\begin{equation*}
B(\Lambda)=\left\{x \in \mathbb{P}^{2}: \operatorname{dist}\left(f^{n}(x), \Lambda\right) \rightarrow 0, \text{as}\;n \rightarrow+\infty\right\}.
\end{equation*}
An attracting basin in $\mathbb{P}^{2}$ is always disjoint from $J_{2},$ except when $\Lambda=\mathbb{P}^{2},$ see \cite{taflin2018attracting} Proposition 1.1 for a proof.
\medskip
\par Fornaess and Sibony \cite{Fornaess1994complex} proved that for any fixed Riemannian metric on $\mathbb{P}^{2},$ if $\Lambda$ is an invariant critical component, then for $ x \in \mathbb{P}^{2}$, when $\text{dist}\;(x,\Lambda) \to 0$, we have
\begin{equation*}
\operatorname{dist}(f(x), \Lambda)=o(\operatorname{dist}(x, \Lambda))
\end{equation*}
\par It follows that for $\epsilon>0$ sufficiently small, the $\epsilon$ -neighborhood 
\begin{equation*}
U_{\epsilon}=\left\{x \in \mathbb{P}^{2}: \operatorname{dist}(x, \Lambda)<\epsilon\right\} 
\end{equation*}satisfies $f\left(U_{\epsilon}\right) \subset \subset U_{\epsilon},$ and $\Lambda=\bigcap_{n \geq 1} f^{n}\left(U_{\epsilon}\right)$. Then $\Lambda$ is an attracting set with trapping region $U_{\epsilon}$.
Fornaess and Sibony \cite{Fornaess1994complex} also showed that an invariant irreducible curve in $\mathbb{P}^{2}$ has genus 0 or
1. Bonifant, Dabija and Milnor showed that an elliptic curve can not be an attracting set \cite{bonifant2007},
so the invariant critical component $\Lambda$ must be a (possibly singular) rational curve.
\medskip
\par Now let $f$ be a PCF map on $\mathbb{P}^{2}$ with an invariant critical component $\Lambda.$ Let
$\pi: \widehat{\Lambda} \rightarrow \Lambda$ be the normalization of $\Lambda,$ then $f$ restricted to $\Lambda$ lifts to a map $\hat{f}$ from $\widehat{\Lambda}$ to itself.
By \cite{jonsson1998some}, $\hat{f}$ is a PCF map on $\mathbb{P}^{1}$ of degree $\geq 2 .$ Let $J(f)$ be the Julia set of
$f,$ and let $\hat{\nu}$ be the unique measure of maximal entropy of $\hat{f} .$ Then $\nu=\pi_{*}(\widehat{\nu})$ is an invariant measure on $\mathbb{P}^{2}$. It follows that Supp $\nu=\pi(J(f)),$ and $\nu$ is a hyperbolic measure of saddle type. Daurat showed that the Green current $T$ is laminar in the basin $\mathcal{B}(\Lambda)$ and subordinate
to the stable manifolds $\bigcup_{x \in \text { Supp } \nu} W^{s}(x) .$ See \cite{daurat2016hyperbolic} for the proof and for more details (it is easy to verify in our case the trapping region $U_{\epsilon}$ satisfies conditions (Tub) and (SJ) in Daurat's paper when $\epsilon$ sufficiently small). See also Bedford-Jonsson \cite{bedford2000dynamics} for the case when $\Lambda$ is totally invariant.
\medskip

\subsection{Main results}
At this stage we get a nice description of the dynamics   in the attracting basins of critical component cycles for PCF maps on $\mathbb{P}^{2}$. In this paper we will show that except maybe for a possibly  ``small" set (the stable manifolds of {\em sporadic} super-saddle cycles), every point in $J_1\setminus J_2$ is indeed contained in the basin of a critical component cycle, and for  every point in $J_2$, there is no Fatou disk passing through it.   Recall that a fixed point is called super-saddle if it has one 0 eigenvalue and one eigenvalue with modulus larger than 1. A sporadic super-saddle fixed point is by deninition a super-saddle fixed point which is not contained in a critical component cycle. We will see  in Lemma 2.9 that there are at most finitely many   sporadic super-saddle cycles.  The following is our main result.
\begin{theorem}
Let $f$ be a PCF map on $\mathbb{P}^{2}$ of degree $\geq 2$, then: 
\medskip
\par (1) $J_1\setminus J_2$ is contained in the union of the basins of  critical component cycles and  stable manifolds of sporadic super-saddle cycles. 
\medskip
\par (2) For every $x\in J_2$ which is not contained in the stable manifold of a sporadic super-saddle cycle, there is no Fatou disk  containing $x$.
\end{theorem}

\par Here are some remarks of Theorem 1.1. First, for general holomorphic endomorphisms on $\mathbb{P}^2$, there can exist an invariant compact set $E\subset J_2$ with positive entropy such that for every $x\in E$ there exists a Fatou disk containing $x$, see Taflin \cite{taflin2017blenders} Theorem 1.6. 

\par Second, if $f$ has no critical component cycle, by Theorem 1.1 and the fact that the Green current $T$ put no mass on pluri-polar set, we conclude  that the stable manifolds of sporadic super-saddle cycles must be contained in $J_2$, and we have $J_1=J_2$. However if $f$ has critical component cycles, we can not easily conclude that the stable manifolds of sporadic super-saddle cycles are contained in $J_2$ by the previous argument. In fact the following phenomenon can happen: there exist a positive closed current $S$ with continuous local potential  in $\mathbb{P}^2$  and two disjoint open sets $\Omega_1$, $\Omega_2$ such that $\text{Supp}(S|_ {\Omega_1})$ and $\text{Supp}(S|_{\Omega_2})$ intersect along a subvariety. A similar phenomenon appears in a dynamical context in Dujardin-Favre \cite{dujardin2008distribution} Example 6.13, see also Dujardin \cite{dujardin200911} Section 5.1. Their example is the bifurcation current in the parameter space of cubic polynomials.

\par Third, to the author's knowledge, there are no sporadic super-saddle cycles for all known examples of PCF maps on $\mathbb{P}^2$. It is interesting to know whether there exists a PCF map on $\mathbb{P}^2$ carrying a sporadic super-saddle cycle. If the answer is yes, it is also interesting to know whether the stable manifold of this sporadic super-saddle point is contained in $J_2$ or $J_1\setminus J_2$. However, under the additional assumption that all branches of $PC(f)$ are smooth and intersect transversally, the answer of this question is no. 

\medskip
\begin{theorem}
Let $f$ be a PCF maps on $\mathbb{P}^2$ of degree $\geq 2$ such that all branches
of $PC(f)$ are smooth and intersect transversally, then every super-saddle cycle is contained in a critical component cycle, that is, sporadic super-saddle cycles do not exist.
\end{theorem}
\medskip
\par We note that the assumption that all branches of $PC(f)$ are smooth and intersect transversally are satisfied by examples constructed by Crass \cite{crass2005family}, Fornaess-Sibony \cite{Fornaess1992critically} and Koch \cite{koch2013teichmuller}, in fact in these examples $PC(f)$ are union of projective lines. A direct corollary of Theorem 1.1 and 1.2 is the following.
\begin{corollary}
Let $f$ be a PCF maps on $\mathbb{P}^2$ of degree $\geq 2$ such that all branches
of $PC(f)$ are smooth and intersect transversally, then 
\medskip
\par (1) $J_1\setminus J_2$ is contained in  the union of the basins of  critical component cycles
\medskip
\par (2) There is no Fatou disk containing $x$ for every $x\in J_2$. 
\end{corollary}
\medskip
\subsection{Some corollaries of the main results}
In \cite{Fornaess2001some}, Fornaess and Sibony asked several questions about the non-wandering set $\Omega(f)$ of
a holomorphic endomorphism on $\mathbb{P}^2$. The following are the questions:
\medskip
\par Q1: Is $\Omega(f)$ the closure of periodic points and Siegel varieties? (Here a Siegel variety is an irreducible analytic set $X$ of positive dimension such that there exists a subsequence $\left\{n_j\right\}$ such that $f^{n_j}|_X\to Id$.)
\par Q2: Is the closure of the repelling periodic points open in $\Omega(f)$?
\par Q3: Is $J_2$ open in $\Omega(f)$?
\par Q4: Describe $\Omega(f)\setminus J_2$.
\medskip

As a corollary of Theorem 1.1, we can give answers to the above four questions for PCF maps on $\mathbb{P}^2$ with no sporadic super-saddle cycles. 

\begin{theorem}
Let $f$ be a PCF map on $\mathbb{P}^2$ of degree $\geq 2$. Then $J_2$ is the closure of repelling
periodic points. If we  further assume that there is no sporadic super-saddle cycle,  then $\Omega(f)$ is the closure of periodic points, and $\Omega(f)\setminus J_2$ is  the union of super-attracting cycles together with $\cup (C\cap J_1)$, where $C$ ranges over the set of critical component cycles, in particular  $J_2$ is open in $\Omega(f)$.
\end{theorem} 
\medskip
Next, as a direct corollary of Theorem 1.1, we obtain a new proof of the laminarity of the Green current in $J_1\setminus J_2$ for PCF maps on $\mathbb{P}^2$, which was first proved by de Th\'elin \cite{de2005phenomene}.
\begin{corollary}[de Th\'elin]
Let $f$ be a PCF map on $\mathbb{P}^2$ of degree $\geq 2$. Then the Green current $T$ is laminar in $J_1\setminus J_2$.
\end{corollary}
\medskip
\par We note that Fornaess and Sibony have studied the non-wandering set and the problem of laminarity of the Green current for $s-$hyperbolic holomorphic endomorphisms on $\mathbb{P}^2$ ($s-$hyperbolic means Axiom A plus some technical conditions), see \cite{Fornaess1998hyperbolic}.
\medskip
\par Finally, we obtain a charactrization of  expanding invariant compact sets for PCF maps on $\mathbb{P}^2$. By using both  Theorem 1.1 and the methods in the proof of Theorem 1.1, we obtain the following result.
\begin{theorem}
Let $f$ be a PCF map on $\mathbb{P}^2$ of degree $\geq 2$ and let $K$ be an invariant compact set. Then $K$ is expanding if and only if $K$ does not contain critical points.
\end{theorem}
\medskip
We note that for rational maps on $\mathbb{P}^1$, an invariant compact set in the Julia set is expanding if it does not intersect  the closure of the post-critical set. It is an open problem that whether the same is true for holomorphic endomorphisms on $\mathbb{P}^k$, i.e. whether an invariant compact set in $J_k$ is expanding if it does not intersect  the closure of the post-critical set. See Maegawa \cite{maegawa2005fatou} for a discussion of this problem. Theorem 1.6 answer this question for PCF maps on $\mathbb{P}^2$.
\medskip
\par As corollaries of Theorem 1.6 and Theorem 1.1, we obtain characterizations of PCF maps on $\mathbb{P}^2$ which are expanding on $J_2$ or satisfy Axiom A.
\begin{corollary}
Let $f$ be a PCF map on $\mathbb{P}^2$ of degree $\geq 2$. Then $f$ is expanding on $J_2$ if and only if every critical component of $f$ is preperiodic to a  critical component cycle, and $f$ is Axiom A  if and only if $f$ is expanding on $J_2$, and for every critical component cycle $C$, $C\cap J_1$ is a hyperbolic (saddle) set.
\end{corollary}

\subsection{ Outline of the paper.} 
\par The structure of this paper is as follows. Section 2 is devoted to some preliminaires. In particular we recall Ueda's results about Fatou maps and the normality of backward iterates of holomorphic endomorphisms on $\mathbb{P}^{k}$, $k\geq 1$. 
\medskip
\par In section 3 we prove Theorem 1.2 and the following theorem
\begin{theorem}
Let $f$ be a PCF map on $\mathbb{P}^{2}$ of degree $\geq 2$, then $J_{2}$ is the closure of the set of repelling periodic points.
\end{theorem}
Here is a comment of Theorem 1.8. We note that for a general holomorphic endomorphism on $\mathbb{P}^{2}$, repelling periodic point may not be contained in $J_2$. Indeed there exist examples possessing isolated repelling points outside $J_{2},$ see \cite{Fornaess2001dynamics} and \cite{hubbard1994superattractive}.
The proof of Theorem 1.8 is a rather quick consequence of a result of Ueda \cite{veda1998critical}. 
\medskip
\par In section 4 we prove Theorem 1.1, by using the previous results. The proof is divided into two steps. First, we  prove that there exists a Fatou disk containing $x$ for every  $x\in J_1\setminus J_2$ which is not contained in the basin of critical component cycle. Second, we prove that  if  there is a Fatou disk containing $x$ for some $x\in J_1$,  then $x$ must be contained in the basin of a  critical component cycle or in  the stable manifold of a sporadic super-saddle cycle. Theorem 1.1 is a combination of these two steps.
\medskip
\par In section 5 we prove Theorem 1.4, Corollary 1.5, Theorem 1.6 and Corollary 1.7.
\medskip
\par In section 6 we discuss some open problems about the existence of sporadic super-saddle periodic points (for PCF maps on $\mathbb{P}^{2}$) and the possible generalization of Theorem 1.1 to higher dimension.
\medskip

\par {\em Acknowledgements.} I would like to thank my advisor Romain Dujardin for his advice, help and encouragement during the course of this work. I also would like to thank Xavier Buff, Henry de Th\'elin, Van Tu Le, Jasmin Raissy, Matteo Ruggiero, Weixiao Shen and Gabriel Vigny for useful discussion.

\section{Preliminaires}
In this section we recall some results of Ueda that we will use later. We start with the following definitions of Ueda \cite{veda1998critical} Definition 4.5 and \cite{uedacritically} Section 1.

\begin{definition}
Let $f$ be a holomorphic endomorphism on $\mathbb{P}^{k}$ of degree $\geq 2.$ Let $Z$ be a complex analytic space. A holomorphic map $h: Z \rightarrow \mathbb{P}^{k}$ is called a Fatou map if $\left\{f^{n} \circ h\right\}_{n \geq 1}$ is a normal family. A Fatou disk $D \subset \mathbb{P}^{k}$ is an image of a non-constant Fatou map $\phi: \mathbb{D} \rightarrow \mathbb{P}^{k},$ where $\mathbb{D}$ is the unit disk.
\end{definition} 
Note that with this definition, a Fatou disk may be singular.
\medskip
\par The following fundamental result about Fatou maps can be found in Ueda \cite{uedacritically} Proposition 2.1. 
\par Let $\pi:\mathbb{C}^{k+1}\setminus \left\{0\right\}\to \mathbb{P}^k$ be the canonical projection. Given a holomorphic map $\phi:X\to \mathbb{P}^k$, we say a holomorphic map $\Phi:X\to \mathbb{C}^{k+1}\setminus \left\{0\right\} $ is a lift of $\phi$ if $\pi\circ\Phi=\phi$ holds.
\par For a homogeneous regular polynomial endomorphism $F$ on $\mathbb{C}^k$, we define the dynamical Green function as 
\begin{equation*}
G(z):=\lim_{n\to \infty}\frac{1}{d^n} \log \|F^n(z)\|,\;\text{for}\;z\in \mathbb{C}^k.
\end{equation*}
\begin{theorem}
Let $X$ be a complex analytic space. Let $\phi:X\to\mathbb{P}^k$ be a holomorphic map. Then the following are equivalent.
\medskip
\par (1) $\phi$ is a Fatou map.
\medskip
\par (2) $\left\{f^j\circ \phi\right\}$ contains a locally uniformly convergent subsequence.
\medskip
\par (3) If $V$ is an open set  of $X$ and $\Phi_V:V\to \mathbb{C}^{k+1}\setminus \left\{0\right\}$ is a holomorphic lift of $\phi|_V$, then $G\circ \Phi_V$ is pluriharmonic on $V$.
\medskip
\par (4) For any $x\in X$, there exists a neighborhood  $V$ of $x$ and $\Phi_V:V\to \mathbb{C}^{k+1}\setminus \left\{0\right\}$  a holomorphic lift of $\phi|_V$, such that $G\circ \Phi_V$ is identically $0$  on $V$.
\end{theorem}
\medskip
\par Next we introduce the conception of points with bounded ramification, first introduced by Ueda in \cite{veda1998critical} Definition 4.5.
\begin{definition}
Let $f$ be a holomorphic endomorphism on $\mathbb{P}^{k}$ of degree $\geq 2.$ A point $q$ is said to be a point of bounded ramification if the following conditions are satisfied:
\medskip
\par (1) there exists a neighborhood $W$ of $q$ such that $PC(f) \cap W$ is an analytic subset of $W$.
\medskip
\par (2) there exists an integer $m$ such that for every $j>0$ and every $p \in f^{-j}(q),$ we have that ord $\left(f^{j}, p\right) \leq m$.
\end{definition} 
\medskip
In the case $k=2$, we have the following characterization of points of bounded ramification for PCF maps on $\mathbb{P}^{2}$. (See Ueda \cite{veda1998critical} Lemma 5.7.)
\begin{lemma}
Let $f$ be a PCF map on $\mathbb{P}^{2}$ of degree $\geq 2.$ Then the points with unbounded ramification are the union of critical component cycles and critical point cycles.
\end{lemma}
\medskip
\par Next we introduce the following abstract result of Ueda. (See \cite{veda1998critical} Lemma 3.7 and Lemma 3.8.)
\begin{lemma}
Let $X$ be a complex manifold and $D$ an analytic subset of $X$ of codimension 1. Let $x\in X$ and let $X_0$ be a compact neighborhood of $x$ such that the pair $(X_0, X_0\cap D)$ can be triangulated, i.e. there exist simplicial complex $K$ and its subcomplex $L$ whose underlying spaces are $X_0$ and $X_0\cap D$. Let $W$ be the open star of $x$ with respect to the complex $K$. Then for every integer $m\geq 0$, there exists an irreducible normal complex space $Z$ and $\eta:Z\to W$ holomorphic such that:
\medskip
\par (1) $\eta$ is $m$-universal, in the sense that for every $D \cap W$-branched holomorphic covering $h: Y \rightarrow W$ (i.e. the ramification locus of $h$ is contained in $D \cap W$) with sheet number $\leq m,$ there exists a holomorphic map $\gamma: Z \rightarrow Y$ such that $h \circ \gamma=\eta .$ 
\medskip
\par (2) If $h: Y \rightarrow W$ is a $D \cap W$ branched holomorphic covering, then $h^{-1}(x)$ is a single point.
\end{lemma} 
\medskip
\par Specializing to holomorphic endomorphisms on $\mathbb{P}^k$, we get the following corollary:
\begin{corollary}

Let $f$ be a holomorphic endomorphism on $\mathbb{P}^{k}$ of degree $\geq 2.$ Let $x\in \mathbb{P}^k$ be a point with bounded ramification. Then for every neighborhood $W$ of $x$ satisfying the assumption in Lemma 2.5, there exists an irreducible normal complex space $Z,$ and $\eta: Z \rightarrow W$ a $PC(f) \cap W$ branched holomorphic covering map such that if $W_{n}$ denotes a connected component of $f^{-n}(W)$, then there exists a holomorphic map $g_{n}: Z \rightarrow W_{n}$ such that $f^{n} \circ g_{n}=\eta$, i.e. the following diagram is commute.

\[ \begin{tikzcd}
Z \arrow{r}{\eta} \arrow[swap]{d}{g_n} & W\\%
W_n\arrow{ru}{f^n}
\end{tikzcd}
\]
\end{corollary}
The map $g_n$ constructed above can be seen as a kind of inverse branch of $f^n$.
\begin{proof}
Since $x$ has bounded ramification, there exists $m\geq 0$ such that ord $\left(f^{n}, x_n\right) \leq m$ for every $x_n\in f^{-n}(x)$ and $n\geq 0$. Take $X=\mathbb{P}^k$ and $D=PC(f)$ in Lemma 2.5. Then for $W$ satisfying the assumption in Lemma 2.5, there exists an irreducible normal complex space $Z,$ and $\eta: Z \rightarrow W$ a $PC(f) \cap W$ branched holomorphic covering map, satisfy the two conclusion of Lemma 2.5. Let $W_{n}$ denote a connected component of $f^{-n}(W)$, then by Lemma 2.5 (2), $W_n\cap f^{-n}(x)$ contains a single point, thus $f^n:W_n\to W$ has sheet number $\leq m$. By Lemma 2.5 (1), there exists a holomorphic map $g_{n}: Z \rightarrow W_{n}$ such that $f^{n} \circ g_{n}=\eta$. Thus the proof is complete.
\end{proof}
\medskip
\par The sequence $\left\{g_{v}\right\}$ defined in Corollary 2.6 is in fact normal, and any limit map of $\left\{g_{v}\right\}$ is also a Fatou map, by the following result of Ueda (\cite{uedacritically} Theorem 2.4 ). 
\begin{lemma}
Let $f$ be a holomorphic endomorphism on $\mathbb{P}^{k}$ of degree $\geq 2 .$ Let $Z$ be a complex analytic space and $h: Z \rightarrow \mathbb{P}^{k}$ a holomorphic map, for every integer $n$ let $g_{n}: Z \rightarrow \mathbb{P}^{k}$ be the holomorphic map such that $f^{n} \circ g_{n}=h .$ Then $\left\{g_{n}\right\}$ is a normal family. Further more if $\phi$ is a limit map of a subsequence of $\left\{g_{n}\right\},$ then $\phi$ is a Fatou map.
\end{lemma}
\medskip
\par The following lemma is implicitly contained in \cite{uedacritically} Theorem 4.2, which will be important in our proof of the main results.
\begin{lemma}
Let $f$ be a holomorphic endomorphism on $\mathbb{P}^{k}$ of degree $\geq 2$. Let $K$ be compact connected subset of $\mathbb{P}^k$ such that $\left\{f^n|_K\right\}$ contains a uniformly convergent subsequence and $K$ is not a single point . let $\phi$ be the limit map of the uniformly convergent subsequence of $f^n|_K$. Then
\medskip
\par (1) If $\phi$ is a constant map, then $x:=\phi(K)$ is a point of unbounded ramification.
\medskip
\par (2) If $\phi$ is not a constant map and there exists $x\in \phi(K)$ which is a point of bounded ramification.  Then there exists a  Siegel variety $Y$ containing $x$, i.e. an irreducible analytic set of positive dimension such that there is a subsequence $\left\{ n_j\right\}$ such that $f^{n_j}|_{Y}\to Id$.
\end{lemma}
\begin{proof}
The proof is contained in the proof of \cite{uedacritically} Theorem 4.2, except at one point one should replace the branched covering map $\xi:X\to \mathbb{P}^k$ in \cite{uedacritically} by the branched covering $\eta:Z\to W$,  where $W$ is a neighborhood of $x$, and $\eta:Z\to W$ is the branched covering constructed in Corollary 2.6. The idea is to  construct  locally uniformly convergent subsequences in the following style:  $\left\{ g_{j}:Z\to \mathbb{P}^k\right\}$ such that $f^{n_j}\circ g_j=\eta$. The property of the limit map  of well-chosen $\left\{ g_j\right\}$ will either give a contradiction of (1), or give the desired Fatou map of (2).
\end{proof}
\medskip
\par We end this section by showing the following finiteness result.
\begin{lemma}
	Let $f$ be a PCF map on $\mathbb{P}^{2}$ of degree $\geq 2$. Then the set of  critical point cycles is a finite set.
\end{lemma}
\begin{proof}
	For every critical point cycle, there is a periodic point $x$ in this cycle such that $x\in C(f)$. Let $V$ be the union of periodic components in $PC(f)$, then $x\in V\cap C(f)$. It is clear that every irreducible curve in $V\cap C(f)$ is contained in a  critical component cycle. So if $x$ is not contained in a critical component cycle, then $x$ is contained in the dimension 0 subvariety of $V\cap C(f)$. Since this dimension 0 part is just a union of finite number of points, there are at most finitely many  critical point cycles which are not contained in a critical component cycle.
\end{proof}

\medskip
\section{Location of periodic points.}
Let $f$ be a PCF map on $\mathbb{P}^{2}$ of degree $\geq 2.$ In this section we prove Theorem 1.2 and Theorem 1.8. By Briend-Duval \cite{briend1999exposants}, $J_{2}$ is contained in the closure of the set of repelling periodic points. Thus to prove $J_2$ is the closure of the set of repelling periodic points, we only need to prove that repelling periodic points are contained in $J_2$.  We start with a definition.

\begin{definition}
Let $f$ be a PCF map on $\mathbb{P}^{2}$ of degree $\geq 2.$ A fixed point $x_{0}$ is called repelling if all eigenvalues of $D f$ at $x_{0}$ have modulus larger than $1.$ A fixed point $x_{0}$ is called super-saddle if $D f$ at $x_{0}$ has one 0 eigenvalue and one eigenvalue with modulus larger than 1. A fixed point $x_{0}$ is called super-attracting if $D f$ at $x_{0}$ has only 0 eigenvalues.
\end{definition} 
Note that by the result of Le \cite{le2019fixed}, for PCF maps on $\mathbb{P}^{2}$ every periodic point is either repelling, super-saddle or super-attracting.
\subsection{Repelling points} 

We first recall the following fundamental result, for the proof see Sibony \cite{sibony1999dynamique} Corollaire 3.6.5.
\begin{proposition}
Let $f$ be a holomorphic endomorphism on $\mathbb{P}^{k}$ of degree $\geq 2$. Then $x \in J_{k}$
if and only if for every neighborhood $U$ of $x$, $\mathbb{P}^{k} \setminus \bigcup_{n=0}^{\infty} f^{n}(U)$ is a pluri-polar set.
\end{proposition}
\medskip
Now we can prove Theorem 1.8.
\begin{theorem}
Let $f$ be a PCF map on $\mathbb{P}^{2}$ of degree $\geq 2$, then every repelling periodic point belongs to $J_2$.
\end{theorem} 
\begin{proof}
Our argument concerns the backward iterates around a fixed point. This kind of argument has already appeared in Fornaess-Sibony \cite{Fornaess2001dynamics}. Without loss of generality we may assume $x_{0}$ is fixed. Since $x_{0} \notin C(f), $ by Lemma 2.4, $x_{0}$ is a point of bounded ramification. We are going to show that for every neighborhood $U$ of
$x_{0}$, we have $\mathbb{P}^{2} \setminus PC(f) \subset \bigcup_{n=0}^{\infty} f^{n}(U).$ Thus since $PC(f)$ is algebraic, by Proposition 3.2 we will get $x_{0} \in J_{2}$.
\medskip
\par Let $y \in \mathbb{P}^{2} \setminus PC(f)$ be an arbitrary point. Let $W$ be a neighborhood of $x_{0}$ such that $y\in W$ and $W$ satisfies the condition in Corollary 2.6. It can be achieved, by first joining $x_{0}$ and $y$ by a smooth embedded (real) curve, and let $W$ be a sufficiently thin tubular neighborhood of this curve such that $\partial W$ is smooth. Let $m\geq 0$ such that ord $\left(f^{n}, x_n\right) \leq m$ for every $x_n\in f^{-n}(x_0)$ and $n\geq 0$. Let $\eta: Z \rightarrow W$ a $PC(f) \cap W$ branched holomorphic covering map as in Corollary 2.6. For $n \geq 0$, let $W_{n}$ denote the connected component of $f^{-n}(W)$ containing $x_{0}.$ Then by Corollary 2.6 we can define holomorphic map $g_{n}: Z \rightarrow W_{n}$ such that $f^{n} \circ g_{n}=\eta.$ By Lemma 2.7, $\left\{g_{n}\right\}$ is a normal family. We are going to show that actually $g_{n}$ converges to the fixed point $x_{0}$. Since $x_{0}$ is repelling, there exists a small neighborhood $\Omega \subset W$ such that $g_{n}$ converges to $x_{0}$ uniformly on $\eta^{-1}(\Omega).$ Now let $\phi$ be any limit map of some subsequence of $\left\{g_{n}\right\}$. Since $\phi$ is constant on an open set $\eta^{-1}(\Omega)$, $\phi$ is constant on $Z$. Thus any limit map of some subsequence of $\left\{g_{n}\right\}$ is the constant map $z\mapsto x_0$. This implies that $g_{n}$ converges to the fixed point $x_{0}$. In particular if $z_0$ satisfies $\eta(z_0)=y$, we have $g_n(z_0)\to x_0$ when $n\to +\infty$.
\medskip
\par Now let $U$ be an arbitrary neighborhood of $x_{0}.$ Since $g_n(z_0)\to x_0$, there exists $N>0$ such that $g_{N}\left(z_0\right) \in U$. Since $f^{N} \circ g_{N}=\eta$ we get $y \in f^{N}(U).$ Since $y\in \mathbb{P}^{2} \setminus PC(f)$ is arbitrary we get $\mathbb{P}^{2} \setminus PC(f) \subset \bigcup_{n=0}^{\infty} f^{n}(U).$ By the arbitrariness of $U$, we have $x_{0} \in J_{2}$, which completes the proof.
\end{proof}
\medskip

\subsection{Super-saddle points.} Let $f$ be a PCF map on $\mathbb{P}^{2}$ of degree $\geq 2$. We investigate the location of super-saddle cycles of $f.$ The following is an equivalent formulation of Theorem 1.2.

\begin{theorem}
Let $f$ be a PCF map on $\mathbb{P}^{2}$ of degree $\geq 2,$ Let $x_{0}$ be a super-saddle fixed point in $\mathbb{P}^{2}$. If the branches of $PC(f)$ are smooth and intersect transversally at $x_{0},$ then $x_{0}$ is contained in a critical component cycle.
\end{theorem} 

\begin{proof}
In the following we let $C_{x_0}(f)$ denote the union of critical components containing $x_{0}$ and let $P C_{x_0}(f):=\bigcup_{n=1}^{\infty} f^{n}(C_{x_0}(f)) .$ It is easy to observe that $P C_{x_0}(f)=P C_{x_0}\left(f^{n}\right)$ for every $n \geq 1$. Then up to an iterate of $f,$ we may assume that all periodic components of $P C_{x_0}(f)$ are invariant, and every component of $PC_{x_0}(f)$ is mapped to an invariant component by at most one iterate. It remains to show that $x_{0}$ is contained in an invariant component of $C_{x_0}(f)$. We argue by contradiction. Assume $x_0$ is not  contained in an invariant component of $C_{x_0}(f)$, then there exists $V$ an invariant component of $PC_{x_0}(f)$ such that $V\not\subset C_{x_0}(f)$. We first show that $V\neq PC_{x_0}(f)$. We argue by contradiction, assuming $V=PC_{x_0}(f)$, then there exists a neighborhood $U$ of $x_0$ such that $f:U\to f(U)$ is a $V$-branched covering. Then, since $x_0\in C_{x_0}(f)$, $f^{-1}(V)$ contains a component of $C_{x_0}(f)$. Since $f^{-1}(V)$ also contains $V$, we deduce that $f^{-1}(V)$ is singular at $x_0$. We recall the following result of Ueda \cite{veda1998critical} Lemma 3.5.
\begin{lemma}
Let $f:U_1\to U_2$ be a $V$-branched holomorphic covering, where $U_1$, $U_2$ are complex manifolds and $V$ is a codimension 1 analytic subset of $U_2$. Suppose that $x_0\in U_1$ is a singular point of $f^{-1}(V)$, then $f(x_0)$ is a singular point of $V$.
\end{lemma}
Coming back to our situation, letting $U=U_1$ and $f(U)=U_2$ in above lemma, we know that $x_{0}$ is a singular point of $V$. This is impossible, since by our assumption $V$ is smooth at $x_0$. Thus there must exist a component $V_{1}$ of $PC_{x_0}(f)$ such that $V_{1} \neq V.$ 
\medskip
\par Let $V^{\prime}=f(V)$, then by our assumption $V'$ is invariant. We recall the following result of Le \cite{le2019fixed} Proposition 5.5.
\begin{lemma}
Let $f:(\mathbb{C}^2,0)\to (\mathbb{C}^2,0)$ be a proper holomorphic germ and let $\Sigma_1, \Sigma_2$ be two irreducible germs at $0$ such that $\Sigma_1\neq \Sigma_2$, $f(\Sigma_1)=\Sigma_2$ and $f(\Sigma_2)=\Sigma_2$. If $\Sigma_2$ is smooth at $0$ then the eigenvalue of $Df$ at $0$ are $0$ and $\lambda$ where $\lambda$ is the eigenvalue of $D_0 f|_{T_0 \Sigma_2}$.
\end{lemma}
Coming back to our situation, take $V=\Sigma_2$ (resp. $V'=\Sigma_2$) in above lemma, since $x_0$ is super-saddle, we get that the eigenvalue of $f$ restricted to $V$ (resp. $V^{\prime}$) at $x_{0}$ must be repelling. By our assumption $V^{\prime}$ and $V$ intersect transversally, the only possible case that $x_{0}$ being super-saddle is when $V^{\prime}=V$. Thus we must have $f\left(V_{1}\right)=V$ by our assumption that every component of $PC_{x_0}(f)$ is mapped to an invariant component by one iterate. To summarize, there exists $V_1\subset P C_{x_0}(f)$ such that $V_1\neq V$, and every component of $P C_{x_0}(f)$ is mapped to $V$ by one iterate (in particular, $f(V_1)=V$).
\medskip
\par We do a local coordinate change such that $x_{0}=(0,0), V=\{x=0\} \cap U,$ where $U$ is a small neighborhood of $(0,0)$. By \cite{le2019fixed} Proposition 5.5, $0$ is a repelling point of $\left.f\right|_{V}$. Let $\lambda$ be the eigenvalue of $f|_V$ at 0, $|\lambda|>1.$ Thus after a linear coordinate change fixing $V$, the expression of $f$ in this coordinate has the following form
\begin{equation}
f(x, y)=(G(x, y), \lambda y+a x+H(x, y)), 
\end{equation}
Where $G(x, y)=O\left(\left|x^{2}\right|,\left|y^{2}\right|\right)$, $x$ is a factor of $G$ and $H(x, y)=O\left(|x|^{2},|y|^{2}\right)$.
\medskip
\par Since $V_{1}$ is a smooth curve that intersects $\{x=0\}$ transversally, we may assume $V_{1}=$ $\{y=\phi(x)\}$ for some holomorphic function $\phi.$ We do a local coordinate change $x^{\prime} \rightarrow x$ $y^{\prime} \rightarrow y-\phi(x).$ In this coordinate the expression of $f$ has the same form as in (3.1), and $V_{1}=\{y=0\}.$ In the following we work in this coordinate.
\medskip

\par By our assumption that every component of $P C_{x_0}(f)$ is mapped to an invariant component by one iterate, there exists a critical component $C$ such that $f(C)=V_{1}$. Thus $C$ satisfies
the equation
\begin{equation}
  \lambda y+a x+H(x, y)=0.
\end{equation}

By the implicit function theorem, $C$ is a smooth curve that intersects with $\{x=0\}$ transversally. We let $C=\{y=\psi(x)\}$ for some holomorphic function $\psi$.
\medskip
\par By direct calculation, the Jacobian of $f$ is
\begin{equation*}
\operatorname{Jac}(f)=\frac{\partial G}{\partial x}\left(\lambda+\frac{\partial H}{\partial y}\right)-\frac{\partial G}{\partial y}\left(a+\frac{\partial H}{\partial x}\right).
\end{equation*}

Since $C$ is in the critical set of $f, \psi$ satisfies the following equation
\begin{equation}
 \quad \frac{\partial G}{\partial x}(x, \psi(x))\left(\lambda+\frac{\partial H}{\partial y}(x, \psi(x))\right)-\frac{\partial G}{\partial y}(x, \psi(x))\left(a+\frac{\partial H}{\partial x}(x, \psi(x))\right)=0. 
\end{equation} 

Take differential of $x$ in the both sides of (3.2) we get
\begin{equation}
\quad \lambda \psi^{\prime}(x)+a+\frac{\partial H}{\partial x}(x, \psi(x))+\frac{\partial H}{\partial y}(x, \psi(x)) \psi^{\prime}(x)=0.  
\end{equation} 

Combining (3.3) and (3.4) we get
\begin{equation*}
\left(\lambda+\frac{\partial H}{\partial y}(x, \psi(x))\right)\left(\frac{\partial G}{\partial x}(x, \psi(x))+\frac{\partial G}{\partial y}(x, \psi(x)) \psi^{\prime}(x)\right)=0.
\end{equation*}

Since $\lambda+\frac{\partial H}{\partial y}(x, \psi(x))\neq 0$ in a neighborhood of $0$, $\psi$ satisfies
\begin{equation*}
\frac{\partial G}{\partial x}(x, \psi(x))+\frac{\partial G}{\partial y}(x, \psi(x)) \psi^{\prime}(x)=0.
\end{equation*}

This implies that $\psi$ satisfies $G(x, \psi(x))=0$ for every $x$. Then by the expression (3.1), we must have $f(C) \subset V.$ Since $f(C)=V_{1},$ we have $f(C)=(0,0),$ which is a contradiction since $f$ is a locally finite to one map. Thus $x_{0}$ is contained in an invariant critical component, and the proof is complete.
\end{proof}
\medskip

\begin{corollary}
Let $f$ be a PCF map on $\mathbb{P}^2$ of degree $\geq 2$ such that all branches of $PC(f)$ are smooth and intersect transversally, then every periodic point in $J_2$ is repelling.
\end{corollary}
\begin{proof}
Let $x_0\in J_2$ be a periodic point. By Le \cite{le2019fixed}, $x_0$ is either repelling, super-saddle or super-attracting. Since super-attracting periodic points belong to the Fatou set, $x_0$ is not super-attracting. By Theorem 3.4, super-saddle periodic points are contained in critical component cycles, in particular they are not in $J_2$. So $x_0$ must be repelling.
\end{proof}
\medskip
\section{Structure of the Julia sets.}
In this section we prove Theorem 1.1. We will first prove that  there exists a Fatou disk containing $x$ for every  $x\in J_1\setminus J_2$ which is not contained in a basin of critical component cycle (Theorem 4.5). Second we prove that  if  there is a Fatou disk containing $x$ for $x\in J_1$,  then $x$ must be contained in the basin of a  critical component cycle or in  the stable manifold of a sporadic super-saddle cycle (Theorem 4.8). Then Theorem 1.1 is a combination of these two theorems. We start with several lemmas.
\begin{lemma}
Let $f$ be a PCF map on $\mathbb{P}^{2}$ of degree $\geq 2.$ Let $x_{0} \in \mathbb{P}^{k}$ and let $v$ be a subsequence of integers such that $x_{v} \rightarrow y$ and $y$ is of bounded ramification, where $x_v=f^v(x_0)$. Let $W$ be a neighborhood of $y$ and $\eta: Z \rightarrow W$ a holomorphic covering as in Corollary 2.6. Let $W_{v}$ denote the connected component of $f^{-v}(W)$ containing $x_{0} .$ Let $g_{v}: Z \rightarrow W_{v}$ such that $f^{v} \circ g_{v}=\eta .$ Assume $g_{v}$ converges to a constant map, then $x_{0} \in J_{2}$ and $y \in J_{2}$
\end{lemma} 
\begin{proof}
We first prove $y\in J_{2}.$ We take $W=B(y,r)$ for sufficiently small $r.$ Let $Z^{\prime}=\eta^{-1}(B(y, r/2)).$ Let $N$ large enough such that $x_{v} \in B(y, r/2)$ when $v \geq N.$ Since $g_{v}$ converges to a constant map and $Z^{\prime} \subset\subset Z$, we have diam $g_{v}\left(Z^{\prime}\right) \rightarrow 0.$ Let $v$ large enough such that $W^{\prime}:=f^{N}\left(g_{v}\left(Z^{\prime}\right)\right) \subset\subset W.$ Then $f^{v-N}: W^{\prime} \rightarrow W$ is a polynomial-like map, By \cite{dinh2010dynamics} Theorem
2.22, there exists a fixed point of $f^{v-N}$ in $W$. Letting $r \rightarrow 0$, we get that $y$ is approximated by periodic points. By the result of Le \cite{le2019fixed}, every periodic point of a PCF map $f$ on $\mathbb{P}^{2}$ is repelling, super-saddle or super-attracting. Since $y$ is not in a critical component cycle, and since there are only finitely many super-attracting periodic points and super-saddle periodic points outside critical component cycles (see Lemma 2.9), $y$ is approximated by repelling periodic points. By Theorem 3.3, $y \in J_{2}.$ Let $z=\eta^{-1}(y)$. The convergence of $ g_{v}$ to $x_{0}$ implies $g_{v}(y) \rightarrow x_{0},$ thus there is a sequence of preimage $\left\{y_{v}\right\}$ of $y$ such that $y_{v}$ converges to $x_{0}.$ By the backward invariance of
$J_{2},$ we conclude that $x_{0} \in J_{2}$.
\end{proof}
\medskip
\begin{lemma}
Let $f$ be a PCF map on $\mathbb{P}^{2}$ of degree $\geq 2 .$ Let $x_{0} \in \mathbb{P}^{2}$ and let $v$ be a subsequence of integers such that $x_{v} \rightarrow y$ and $y$ is of bounded ramification, where $x_v=f^v(x_0)$. Let $W$ be a neighborhood of $y$ and $\eta: Z \rightarrow W$ a holomorphic covering as in Corollary 2.6. Let $W_{v}$ denotes the connected component of $f^{-v}(W)$ containing $x_{0} .$ Let $g_{v}: Z \rightarrow W_{v}$ such that $f^{v} \circ g_{v}=\eta .$ Assume $g_{v}$ converges to a non-constant map $\phi,$ then there exists a Fatou disk containing $x_{0}$.
\end{lemma}

\begin{proof}
Let $W=B(y,r)$ for small $r$. Let $M=\phi(Z),$ we will show that $M$ contains a Fatou disk containing $x_{0}.$ Let $N$ large enough such that $x_{v} \in W$ when $v \geq N.$ Then for $v \geq N,$ there exist $z_{v} \in Z^{\prime}$ such that $g_{v}\left(z_{v}\right)=x_{0}.$ Let $z=\eta^{-1}(y),$ it is clear that $z_{v} \rightarrow z$ when $v \rightarrow+\infty.$ Let $v \rightarrow+\infty$ in the equation $g_{v}\left(z_{v}\right)=x_{0}$ we get $\phi(z)=x_{0},$ then $x_{0} \in M$. By Lemma 2.7, $\phi: Z \rightarrow \mathbb{P}^{2}$ is a Fatou map, by definition, this implies that $\left\{\left.f^{n}\right|_{M}\right\}_{n \geq 1}$ is a normal family. Let $D \subset Z$ be a holomorphic disk containing $z$ such that $\phi$ is not a constant map when restricted to $D,$ then $\phi(D) \subset M$ is a Fatou disk containing $x_{0}.$
\end{proof} 
\medskip
\begin{lemma}
Let $f$ be a PCF map on $\mathbb{P}^{2}$ of degree $\geq 2 .$ Let $x_{0} \in J_{1}$ such that $x_{0}$ is not  contained in the attracting basin of a critical component cycle nor contained in the stable manifold of a super-saddle cycle, then there exists a subsequence $v$ of positive integers such that $x_{v}=f^{v}\left(x_{0}\right) \rightarrow y,$ where $y$ is a point of bounded ramification.
\end{lemma} 

\begin{proof}
There are at most finitely many critical point cycles, which are not contained in the critical component cycles (see Lemma 2.9). We denote this finite set by $E.$ We show that if $x_{0} \in J_1$ satisfying  that $x_{0}$ is not contained in the attracting basins of critical component cycles and $\omega\left(x_{0}\right)$ contains only points of unbounded ramification, then $x_0$ is contained in the stable manifold of a sporadic super-saddle cycle. If $x_0$ satisfies the above assumption, by Lemma 2.4 we know that $\omega\left(x_{0}\right) \subset E.$ We recall the following basic property of $\omega$ -limit set: 
\begin{lemma}
Let $X$ be a compact metric space, let $f:X\to X$ be a continuous map, and let $g:\omega(x_0)\to\omega(x_0)$ be the restriction of $f$ on the $\omega$-limit set of $x_0$, then there is no non-trivial open subset $U$ of $\omega(x_0)$ such that $g(\overline{U})\subset U$.
\end{lemma} 
\begin{proof}
Bowen (\cite{bowen1975omega} Theorem 1) proved the above lemma for homeomorphisms, but the proof also holds for non-invertible maps. For the completion we give a proof here. Assume by contradiction that such open subset $U$ exists, let $Y:=\omega(x_0)$. By our assumption $2\epsilon:=\text{dist}(Y\setminus U, g(\overline{U}))>0$. Choose $0<\delta<\epsilon$ such that $\text{dist}(x_1,x_2)<\delta$ implies $\text{dist}(f(x_1),f(x_2))<\epsilon$ for every $x_1,x_2\in X$. Now it is clear that there is $N>0$ such that $\text{dist}(f^n(x_0),Y)<\delta$ when $n>N$ (otherwise $\omega(x_0)$ will be strictly larger than $Y$). Pick $M\geq N$ such that $\text{dist}(f^M(x_0), g(\overline{U}))<\epsilon$ and $\text{dist}(f^M(x_0), y)<\delta$ for some $y\in Y$. Then $\text{dist}(g(\overline{U}), y)<2\epsilon$, which implies $y\in U$. Then
\begin{equation*}
  \text{dist}(f^{M+1}(x_0), g(\overline{U}))\leq \text{dist}(f^{M+1}(x_0), g(y))<\epsilon.
\end{equation*}
Inductively for all $m\geq M$ we have $\text{dist}(f^{m}(x_0), g(\overline{U}))<\epsilon$. This implies $Y\cap (Y\setminus U)=\emptyset$, which is a contradiction.
\end{proof}
\par Coming back to our situation, since $\omega(x_0)$ is a finite set, the only possibility is that $\omega\left(x_{0}\right)$ contains a single periodic cycle. Since $x_{0} \in J_1$ and $\omega(x_0)$ contains no points of bounded ramification, this cycle must be a super-saddle cycle. Thus $x_{0}$ is contained in the stable manifold of this super-saddle cycle. It follows that if $x_{0}$ is not contained in the attracting basin of a critical component cycle nor contained in the stable manifold of a super-saddle cycle, $\omega\left(x_{0}\right)$ has non-empty intersection with the set of points of bounded ramification. This completes the proof.
\end{proof} 
\medskip

Now we can prove the following theorem.

\begin{theorem}
Let $f$ be a PCF map on $\mathbb{P}^{2}$ of degree $\geq 2 .$ If $x \in J_{1} \setminus J_{2}$ which is not contained in the basin of a critical component cycle, then there is a Fatou disk $D$ containing $x$.
\end{theorem} 

\begin{proof}
Suppose first $x_{0}$ is contained in the stable manifold of a super-saddle cycle. Then there exists an embedded holomorphic disc $D$ containing $x_{0}$ such that $D$ coincides with the local stable manifold at $x_{0},$ and it is clear that $\left\{f^{n} | D\right\}_{n \geq 1}$ is normal. Now suppose $x_{0}$ satisfies the assumptions of the theorem, and $x_{0}$ is not contained in the stable manifold of a super-saddle cycle. By Lemma $4.3,$ we can choose a subsequence $v$ of positive integers such that $x_{v}=f^{v}\left(x_{0}\right) \rightarrow y,$ and $y$ is of bounded ramification. Consider a neighborhood $W$ of $y$ and $\eta: Z \rightarrow W$ the holomorphic covering as in Corollary 2.6. Let $W_{v}$ denotes the connected component of $f^{-v}(W)$ containing $x_{0} .$ Let $g_{v}: Z \rightarrow W_{v}$ such that $f^{v} \circ g_{v}=\eta .$ By Lemma 2.7 $\left\{g_{v}\right\}$ is a normal family. By passing to some subsequence, we may assume $g_{v}$ converges to a holomorphic map $\phi$. $\phi$ can not be a constant map, since otherwise by Lemma 4.1, $x_{0} \in J_{2},$ which is a contradiction. Thus $\phi$ is a non-constant map. By Lemma $4.2,$ there is a Fatou disk containing $x_{0},$ which completes the proof.
\end{proof}
\medskip
Before proving Theorem 4.8, we need some lemmas. 
\begin{lemma}
Let $f$ be a PCF map on $\mathbb{P}^{2}$ of degree $\geq 2 .$ Let $x \in J_1$ which is not contained in the  basin of a critical component cycle, nor in the stable manifold of a super-saddle cycle. Assume moreover that there is a Fatou disk $X$ containing $x$. Then there is a point $y\in \omega(x)$ such that there is a Siegel variety containing $y$.
\end{lemma}
\begin{proof}
 We choose $y\in \omega(x)$ such that $y$ is not a super-saddle periodic point. Then by Lemma 2.4, $y$ is a point of bounded ramification. Let $\phi$ be a limit map of $\left\{ f^n|_X\right\}$ such that $\phi(x)=y$. Then by Lemma 2.8, there is a Siegel variety $X$ containing $y$.
\end{proof}
\medskip

\begin{lemma}
let $f$ be a PCF map on $\mathbb{P}^2$ of degree $\geq 2$. Let $X$ be a Siegel variety, then $X\not\subset PC(f)$.
\end{lemma}
\begin{proof}
Assume by contradiction that $X\subset PC(f)$. By passing to an iterate of $f$ we may assume $X$ is contained in an invariant component $V$ of $PC(f)$. Let $\pi: \widehat{V} \rightarrow V$ be the normalization of
$V,$ and let $\hat{f}$ be the lift of $f$ on $\widehat{V}$. Then by Jonsson \cite{jonsson1998some} $\widehat{V}=\mathbb{P}^1$ or a torus, and $\hat{f}$ is a PCF map on $\mathbb{P}^1$ or an expanding torus map. We have the following commutative diagram
\[ \begin{tikzcd}
\widehat{V} \arrow{r}{\hat{f}} \arrow[swap]{d}{\pi} & \widehat{V} \arrow{d}{\pi} \\%
V \arrow{r}{f}& V
\end{tikzcd}
\] 
\par Let $U\subset \pi^{-1}(X)$ be an open set which is contained in the unramification locus of $\pi$. Then $U$ is a Siegel variety in $\widehat{V}$. It is well-known that a PCF map on $\mathbb{P}^1$ or a expanding torus map can not have such a Siegel variety, a contradiction.
\end{proof}
\medskip
\par Next we can prove Theorem 4.8, thus completes the proof of Theorem 1.1.
\begin{theorem}
Let $f$ be a PCF map on $\mathbb{P}^{2}$ of degree $\geq 2$. If $x\in J_1$ such that there is a Fatou disk containing $x$. Then $x$ must be contained  in the basin of a critical component cycle, or in the stable manifold of a super-saddle cycle.
\end{theorem}
\begin{proof}
Our argument borrows some ingredients in Ueda's paper \cite{uedacritically} and Le's paper \cite{le2019fixed}. We argue by contradiction. Suppose $x$ is not contained  in the basin of a critical component cycle, nor in the stable manifold of a super-saddle cycle. Then by Lemma 4.6, there is a point $y\in \omega(x)$ such that there is a Siegel variety $X$ containing $y$. By Lemma 4.7 we may assume $X\cap PC(f)=\emptyset$ (shrink $X$ if necessary).
\medskip

\par Let $Z$ be the universal covering space of $\mathbb{P}^2\setminus PC(f)$, and let $\eta:Z\to \mathbb{P}^2\setminus PC(f)$ be the universal covering map. Let $z_0\in Z$ satisfying $\eta(z_0)\in \phi(X)$. Recall that $\phi=\lim_{j\to\infty} f^{n_j}|_X$. By passing to a subsequence we may assume $k_j:=n_{j+1}-n_j\to+\infty$. Let $W$ be a neighborhood of $\eta(z_0)$ such that $W\cap PC(f)=\emptyset$, and let $Z_0$ be a neighborhood of $z_0$ such that $\eta:Z_0\to W$ is a biholomorphism.  Let $x_0\in X$ satisfying $f^{n_j}(x_0)\to \eta(z_0)$. Let $W_j$ be the connected component of $f^{-k_j}(W)$ containing $f^{n_j}(x_0)$. Let $g_j:Z_0\to W_j$ be the holomorphic map such that $f^{k_j}\circ g_j=\eta$ as in Corollary 2.6. Since $Z$ is simply connected, there is no obstruction to extend  $g_j$ as a globally defined map on $Z$. In the following $g_j$ will denote this globally defined map. By Lemma 2.7 $\left\{g_j\right\}$ is a  normal family. By shifting to a subsequence we may assume $\left\{g_j\right\}$ converges locally uniformly to a holomorphic map $g$.
\medskip
\par Let $V:=\eta^{-1}(X)\cap Z_0$. We claim that $g=\eta$ on $V.$ To show this, let $z\in V$ be an arbitrary point, we are going to show $g(z)=\eta(z)$. Let $x\in X$ satisfying $f^{n_j}(x)\to \eta(z)$.  Since for $j$ large enough we have $f^{n_{j+1}}(x)\in W$, we get $f^{n_j}(x)\in W_j$. Let $z_{j+1}\in Z_0$ such that $g_j(z_{j+1})=f^{n_j}(x)$.  By the equation $f^{k_j}\circ g_j=\eta$ we have that $\eta(z_{j+1})=f^{n_{j+1}}(x)$. Since $f^{n_{j+1}}(x) \to \eta(z)$, we get that $z_{j+1}\to z$. Now let $j\to+\infty$ in the equation $g_j(z_{j+1})=f^{n_j}(x)$ we get that $g(z)=\eta(z)$. Thus $g=\eta$ on $V.$
\medskip
\par The equation $g=\eta$ defines a closed analytic set $Y\subset Z$. Since $V\subset Y$, $Y$ has dimension one. Let $Y_0$ be an irreducible component of $Y_0$ of pure dimension 1. By Lemma 2.7 $g$ is a Fatou map, thus $\eta|_{Y_0}$ is a Fatou map since $g=\eta$ on $Y_0$. Since $Y_0$ is a closed analytic set in $Z$ and $\eta$ is a covering map, the following property holds: for every $x\in\mathbb{P}^2\setminus PC(f)$, there is a neighborhood $U$ of $x$ such that either $\eta^{-1}(U)\cap Y_0=\emptyset$ or every connected component of $\eta^{-1}(U)\cap Y_0$ is relatively compact in $Y_0$.
\medskip
\par Let $\pi:\mathbb{C}^{3}\setminus \left\{0\right\}\to \mathbb{P}^2$ be the canonical projection. Let $F$ be a homogeneous polynomial endomorphism on $\mathbb{C}^3$ which induce $f$. It is clear that $F$ is a PCF map on $\mathbb{C}^3$.  Let $G_F$ be the Green function of $F$. Recall that  $\eta:Y_0\to \mathbb{P}^2$ is a Fatou map. By Theorem 2.2, there exists a holomorphic lift  $\Phi:Y_0\to \mathbb{C}^{3}\setminus \left\{0\right\} $ of $\eta|_{Y_0}$ such that $\pi\circ\Phi=\eta$ and $G_F\circ \Phi$ is identically zero on $Y_0$. Since $\eta(Y_0)\cap PC(f)=\emptyset$ we get that $\Phi(Y_0)\cap PC(F)=\emptyset$. By the construction of $\Phi$ and by the fact that $\pi$ is an open map, the following property holds: for every $x\in\mathbb{C}^3\setminus PC(F)$, there is a neighborhood $U$ of $x$ such that either $\Phi^{-1}(U)\cap Y_0=\emptyset$ or every connected component of $\Phi^{-1}(U)\cap Y_0$ is relatively compact in $Y_0$. 
\medskip
\par Let $Y_1$ be the smooth part of $Y_0$, then $Y_1$ is a Riemann surface. Let $S:=\Phi(Y_0\setminus Y_1)$, then $S$ is a discrete set in $\mathbb{C}^3\setminus PC(F)$. Since $G_F\circ \Phi$ is identically zero on $Y_1$, $\Phi(Y_1)$ is contained in a bounded domain in  $\mathbb{C}^3$. This implies that $Y_1$ is a hyperbolic Riemann surface, i.e. there exists a universal covering map $\kappa:\mathbb{D}\to Y_1$, where $\mathbb{D}$ is the unit disk. 
Since $\kappa$ is a covering map, the following property holds: 
\par for every $x\in\mathbb{C}^3\setminus (PC(F)\cup S)$, there is a neighborhood $U$ of $x$ such that either $(\Phi\circ\kappa)^{-1}(U)\cap \mathbb{D}=\emptyset$ or every connected component of $(\Phi\circ\kappa)^{-1}(U)\cap \mathbb{D}$ is relatively compact in $\mathbb{D}$ ($\star$).
\medskip
\par We consider the map $\Phi\circ\kappa: \mathbb{D}\to \mathbb{C}^3$. The following idea of using Fatou-Riesz's theorem was first appeared in Le \cite{le2019fixed}. As we mentioned before $\Phi\circ\kappa (\mathbb{D})$ is contained in a bounded domain in $\mathbb{C}^3$. By the well known Fatou-Riesz's theorem (see Milnor \cite{milnor2011dynamics} Theorem A.3), for Lebesgue a.e. $\theta\in [0,2\pi)$, the radial limit
\begin{equation*}
\tau_\theta=\lim_{r\to 1^-}\Phi\circ\kappa(re^{i\theta})
\end{equation*}
exists for some $\tau_\theta\in\mathbb{C}^3$.
\medskip
\par We next show that for every limit $\tau_\theta$, $\tau_\theta\in PC(F)\cup S$. Assume by contradiction that $\tau_\theta\notin PC(F)\cup S$. Let $U$ be a small neighborhood of $\tau_\theta$, then by property $(\star)$, $(\Phi\circ\kappa)^{-1}(U)$ is a disjoint union of relatively compact open set in $\mathbb{D}$.  Then for every $R\in (0,1)$ there exists a $r$ satisfying $R<r<1$ such that $re^{i\theta}\notin (\Phi\circ\kappa)^{-1}(U)$ (by the connectedness of the interval $[R,1]$). This contradicts the fact that $\tau_\theta=\lim_{r\to 1^-}\Phi\circ\kappa(re^{i\theta})$. Thus $\tau_\theta\in PC(F)\cup S$.
\medskip
\par Since $S$ is a countable set, there must exist a positive Lebesgue measure set $E\subset [0,2\pi)$ such that for every $\theta\in E$ either $\lim_{r\to 1^-}\Phi\circ\kappa(re^{i\theta})\in PC(F)$ or $\lim_{r\to 1^-}\Phi\circ\kappa(re^{i\theta})=s$ for an element $s\in S$. We will show that these two cases are both impossible, then we will get a contradiction. 
\medskip
\par We first treat the  case that for every $\theta\in E$, $\lim_{r\to 1^-}\Phi\circ\kappa(re^{i\theta})\in PC(F)$. Let $P$ be a homogeneous polynomial that defines $PC(F)$, i.e.
\begin{equation*}
 PC(F)=\left\{x\in\mathbb{C}^3: P(x)=0\right\}.   
\end{equation*}
Then $P\circ\Phi\circ\kappa$ is a bounded holomorphic function on $\mathbb{D}$ such that the radial limit 
\begin{equation*}
\lim_{r\to 1^-}P\circ\Phi\circ\kappa(re^{i\theta})=0
\end{equation*}
for every $\theta\in E$, where $E$ has positive Lebesgue measure. By Fatou-Riesz's theorem $P\circ\Phi\circ\kappa$ is identically zero on $\mathbb{D}$. This implies $\Phi\circ \kappa(\mathbb{D})\subset PC(F)$ which imples $\Phi(Y_1)\subset PC(F)$, a contradiction.
\medskip
\par Second we treat the case that for every $\theta\in E$, $\lim_{r\to 1^-}\Phi\circ\kappa(re^{i\theta})=s$, where $s\in S$. Let $P_1$, $P_2$, $P_3$  be three linear polynomial in $\mathbb{C}^3$ that define $s$, i.e.
\begin{equation*}
 \left\{s\right\}=\left\{x\in\mathbb{C}^3: P_i(x)=0, i=1,2,3\right\}.   
\end{equation*}
Then for every $i$  we get that $P_i\circ\Phi\circ\kappa$ is a bounded holomorphic function on $\mathbb{D}$ such that the radial limit 
\begin{equation*}
\lim_{r\to 1^-}P\circ\Phi\circ\kappa(re^{i\theta})=0
\end{equation*}
for every $\theta\in E$, where $E$ has positive Lebesgue measure. By Fatou-Riesz's theorem, for every $i$ we get that $P_i\circ\Phi\circ\kappa$ is identically zero on $\mathbb{D}$. This implies $\Phi\circ \kappa(\mathbb{D})=\left\{s\right\}$ which imples $\Phi(Y_1)=\left\{s\right\}$, contradicts the fact that $Y_1$ has dimension 1. This completes the proof.
\end{proof}

\medskip
\par We end this section by showing the  sporadic super-saddle cycle is contained in $J_2$.
\begin{lemma}
	Let $f$ be a PCF map on $\mathbb{P}^{2}$ of degree $\geq 2$, then every sporadic super-saddle cycle is contained in $J_2$.
\end{lemma} 
\begin{proof}
	Let $x$ be a super-saddle cycle which is not contained in a critical component cycle. Passing to an iterate of $f$ we may assume every $x$ is fixed and every periodic component of $PC(f)$ is invariant. Let $V$ be an invariant component of $PC(f)$ containing $x$, by our assumption $V\not\subset C(f)$. By Lemma 3.6, the eigenvalue $\lambda$ of $Df|_{T_x V}$ satisfies $|\lambda|>1$. Let $\pi: \widehat{V} \rightarrow V$ be the normalization of
	$V,$ and let $\hat{f}$ be the lift of $f$ on $\widehat{V}$. Let $\hat{x}\in \widehat{V}$ such that $\pi(\hat{x})=x$. Since $V$ is smooth at $x$, such a $\hat{x}$ is unique. Then $\hat{x}$ is a repelling fixed point of $\hat{f}$. Since $\widehat{V}$ is either $\mathbb{P}^1$ or a torus and $\hat{f}$ is either a rational function of degree $\geq 2$ or an expanding torus map, $\hat{x}$ can be approximated by a sequence of  periodic points $\left\{ \hat{y}_n\right\}\subset \widehat{V}$ such that $y_n=\pi(\hat{y}_n)$ is in the smooth part of $PC(f)$. Then by the result of Le \cite{le2019fixed}, $y_n$ is a  repelling periodic point of $f$ for every $n$. By Theorem 1.8, $y_n\in J_2$. Since $y_n\to x$ we get $x\in J_2$ as desired.
\end{proof}
\medskip
\section{Some corollaries of the main results}
In this section we prove  Theorem 1.4, Corollary 1.5, Theorem 1.6 and Corollary 1.7.
\medskip
\subsection{Non-wandering set for PCF maps on $\mathbb{P}^2$}
Before proving Theorem 1.4, we prove the following lemma.

\medskip
\begin{lemma}
Let $f$ be a PCF map on $\mathbb{P}^{2}$ of degree $\geq 2 . $ Let $C$ be an invariant critical component.  Then $\Omega(f)\cap C$ is the union of super-attracting cycles and $J_1\cap C$, moreover periodic points are dense in $\Omega(f)\cap C$.
\end{lemma}
\begin{proof}
 Let $\pi: \widehat{C} \rightarrow C$ be the normalization of
$V,$ and let $\hat{f}$ be the lift of $f$ on $\widehat{C}$. Then $\widehat{C}=\mathbb{P}^1$, and $\hat{f}$ is  a rational function of degree $\geq 2$, see \cite{Fornaess1994complex}. Moreover by Jonsson \cite{jonsson1998some}, $\hat{f}$ is also a PCF map.
 We have the following commutative diagram.
\[ \begin{tikzcd}
\widehat{C} \arrow{r}{\hat{f}} \arrow[swap]{d}{\pi} & \widehat{C} \arrow{d}{\pi} \\%
C \arrow{r}{f}& C
\end{tikzcd}
\]
\par It is clear that if  $p$ is a periodic point of $\hat{f},$ then $\pi(p)$ is a periodic point of $f.$ 
 Let $J(\hat{f})$ be the Julia set of $\hat{f}$. Since periodic points of $\hat{f}$  are dense in $J(\hat{f})$, periodic points of $f$ are dense in $\pi(J(\hat{f})$.
  We claim that $J_1\cap C=\pi(J(\hat{f})$. By the above discussion $\pi(J(\hat{f})$ is contained in the closure of periodic points. These periodic points must be super-saddle. Thus $\pi(J(\hat{f})\subset J_1\cap C$. To show $J_1\cap C\subset \pi(J(\hat{f})$, it is equivalent to show that for every $y$ in the Fatou set of $\hat{f}$, $\pi(y)$ is in the Fatou set of $f$. Since $\hat{f}$ is a PCF map, $y$ is contained in the basin of a super-attracting cycle of $\hat{f}$. For every super-attracting fixed point $z$ of $\hat{f}$, $\pi(z)$ is a fixed point of $f$. It is clear that $\pi(z)$ is not repelling, since $\pi(z)\in C(f)$.  It is also clear that $\pi(z)$ is not super-saddle, since otherwise $\pi(z)$ is a smooth point of $C$ (see Le \cite{le2019fixed} Proposition 5.4), and $z$ must be a repelling point of $\hat{f}$. Thus $\pi(z)$ must be  a super-attracting fixed point of $f$. Thus $\pi(y)$ is in the basin of $\pi(z)$, this implies $\pi(y)$ is in the Fatou set of $f$. Thus $J_1\cap C=\pi(J(\hat{f})$ as desired. Since $\pi(J(\hat{f})$ is contained in the closure of periodic points, $J_1\cap C$ is contained in $\Omega(f)$. Since the Fatou set of $f$ is the union of basins of super-attracting cycles, we get that $\Omega(f)\cap C$ is exactly the union of super-attracting cycles and $J_1\cap C$. This implies periodic points are dense in $\Omega(f)\cap C$.
  
\end{proof}

\medskip
\par Now we can prove Theorem 1.4. Recall the statement.
\begin{theorem}
Let $f$ be a PCF map on $\mathbb{P}^2$ of degree $\geq 2$. Then $J_2$ is the closure of repelling
periodic points. If further assumed there is no sporadic super-saddle cycle, then $\Omega(f)$ is the closure of periodic points and $\Omega(f)\setminus J_2$ is  the union of super-attracting cycles together with $\cup (C\cap J_1)$, where $C$ ranges over the set of critical component cycles, in particular  $J_2$ is open in $\Omega(f)$.
\end{theorem} 
\begin{proof}
By Theorem 1.8, $J_2$ is the closure of the set of repelling
periodic points. In the following we assume there is no sporadic super-saddle cycle.  First we show that $\Omega(f)$ is the closure of periodic points. Let $x\in \Omega(f)$ such that $x\notin J_2$. By Theorem 1.1 and the fact that Fatou set are only super-attracting basins, either $x$ is contained in the basin of a super-attracting cycles, or $x$ is contained in the basin of a critical component cycle. In the first case it is clear that $x$ must be a super-attracting periodic point. In the second case, $x$ must be contained in a critical component cycle, thus $x$ is contained in the closure of periodic points by Lemma 5.1. Again by Lemma 5.1, for every critical component $C$ we have $\Omega(f)\cap C$ is the union of super-attracting cycles and $J_1\cap C$. Thus $\Omega(f)\setminus J_2$ is  the union of super-attracting cycles together with $\cup (C\cap J_1)$, where $C$ ranges over the set of critical component cycles.
\end{proof}
\medskip
\subsection{Laminarity of the Green current for PCF maps on $\mathbb{P}^2$}
Recall the statement of Corollary 1.5.
\begin{corollary}[de Th\'elin]
Let $f$ be a PCF map on $\mathbb{P}^2$ of degree $\geq 2$. Then the Green current $T$ is laminar in $J_1\setminus J_2$.
\end{corollary}
\begin{proof}
Let $E$ be the union of sporadic super-saddle cycles. Let $W$ be the union of stable manifolds of super-saddle cycles in $E$. Then $W$ is a pluri-polar set, and $T$ put no mass on $W$ (since the local potential of $T$ is H\"older continuous, see \cite{dinh2010dynamics}). By Theorem 1.1, $T|_{J_1\setminus J_2}$ put full mass on the union  of the basins of critical component cycles. By Daurat \cite{daurat2016hyperbolic}, $T$ is laminar in every basin of critical component cycle. Thus $T$ is laminar in $J_1\setminus J_2$.
\end{proof}
\medskip
\subsection{Expanding sets for PCF maps on $\mathbb{P}^2$}
Recall the statement of Theorem 1.6. The following proof is an adaption of the argument in Maegawa \cite{maegawa2005fatou}.
\begin{theorem}
Let $f$ be a PCF map on $\mathbb{P}^2$ of degree $\geq 2$ and let $K$ be an invariant compact set. Then $K$ is expanding if and only if $K$ does not contain critical points.
\end{theorem}
\begin{proof}
The only if part is obvious. In the following we prove $K$ is expanding if  $K$ does not contain critical points. To prove that $K$ is an expanding set, it is equivalent to prove the following limit tends to infinity:
\begin{equation*}
\min_{x\in K}\min_{v\in T_x\mathbb{P}^2, |v|=1}|D(f^n)(v)|\to+\infty, \; n\to+\infty.
\end{equation*}

\par Assume by contradiction that the limit does not go to infinity. Then there exists a constant $C>0$ and a sequence of points $\left\{x_{m}\right\}\subset K$ and vectors $\left\{v_{m}\right\}\subset T_K\mathbb{P}^2$ such that  $|v_{m}|=1$ and $|D(f^{n_m})_{x_m}(v_m)|<C$, where $\left\{ n_m \right\}$ is a sequence of integers. By shifting to a subsequence we may assume $f^{n_m}(x_m)\to x\in K$. Since $K$ is invariant and $K\cap C(f)=\emptyset$, by Lemma 2.4 we know that $x$ is a point of bounded ramification. Let $W$ be a neighborhood of $x$ and let $\eta:Z\to W$ be the branched covering constructed in Corollary 2.6. Let $W_m$ be the connected component of $f^{-n_m}(W)$ containing $x_m$. Let $g_m:Z\to W_m$ be the holomorphic map such that $f^{n_m}\circ g_m=\eta$. Then by Lemma 2.7 by shifting to a subsequence we may assume $\left\{g_m\right\}$ locally uniformly converges to a Fatou map $g$. Since $K$ is an invariant compact set and $K\cap C(f)=\emptyset$, by Theorem 1.1 we get $K\subset J_2$. Again by Theorem 1.1 and the fact that $K\cap C(f)=\emptyset$, there is no Fatou disk containing points in $K$. Thus $g$ must be a constant map. Thus we have $\text{diam}\;W_m\to 0$. Since $K\cap C(f)=\emptyset$, $W_m\cap C(f)=\emptyset$ for $m$ large enough. Thus $f^{n_m}:W_m\to W$ is a biholomorphism when $m$ large. We can then define $g'_m:W\to W_m$ be the inverse map of $f^{n_m}|_{W_m}$, and it easily follows that $g'_m$ converges locally uniformly to a constant map. This implies
\begin{equation}
\max_{y\in U}\max_{v\in T_y \mathbb{P}^2, |v|=1} |Dg'_m(v)|\to 0,\;m\to +\infty,
\end{equation}
where $U$ is a small neighborhood of $x$. 
\medskip
\par By the definition of $g'_m$ (5.1) is equivalent to the following
\begin{equation}
\min_{z\in U_m }\min_{v\in T_z\mathbb{P}^2, |v|=1}|D(f^{n_m})(v)|\to+\infty, \; m\to+\infty,
\end{equation}
where $U_m:=f^{-n_m}(U)\cap W_m$.
\medskip
\par By our construction of $\left\{x_m\right\}$ and $\left\{v_m\right\}$, it follows  that $x_m\in U_m$ for $m$ large and
\begin{equation*}
  |D(f^{n_m})_{x_m}(v_m)|<C.  
\end{equation*} 
This contradicts (5.2). Thus the proof is complete.
\end{proof}
\medskip
\subsection{PCF maps on $\mathbb{P}^2$ that are expanding on $J_2$ or Axiom A }
We first discuss the PCF maps on $\mathbb{P}^2$ that are expanding on $J_2$. Recall the first part of the statement of Corollary 1.7.
\begin{corollary}
Let $f$ be a PCF map on $\mathbb{P}^2$ of degree $\geq 2$. Then $f$ is expanding on $J_2$ if and only if every critical component of $f$ is preperiodic to a  critical component cycle.
\end{corollary}
\begin{proof}
We first show the if part. If every critical component of $f$ is preperiodic to a  critical component cycle, then $C(f)$ is contained in the basins of critical component cycles. Thus $J_2\cap C(f)=\emptyset$ and $J_2$ is expanding follows form Theorem 1.6.
\medskip
\par Next we show the only if part. Assume by contradiction that there is a critical component $C$ which is not preperiodic to a critical component cycle. Let $V$ be a periodic component of $PC(f)$ such that $V\not\subset C(f)$. Then by the argument in the proof of Lemma 5.1, there is a repelling periodic point $x\in V$. By Theorem 1.8 we know that $x\in J_2$. Let $y\in C(f)$ such that $f^n(y)=x$ for some $n>0$. Then $y\in J_2$ since $J_2$ is totally invariant. Thus $J_2\cap C(f)\neq \emptyset$. Thus $J_2$ is not expanding, which is a contradiction.  This completes the proof.
\end{proof}
\medskip
Next we introduce the notion of Axiom A for holomorphic endomorphisms on $\mathbb{P}^k$. (For the general definition for smooth endomorphisms on Riemannian manifold, see for instance \cite{qian1995ergodic}.)
\begin{definition}
Let $f$ be a holomorphic endomorphism on $\mathbb{P}^2$ of degree $\geq 2$. Then $f$ satisfies the Axiom A if the following two conditions hold:
\medskip
\par (1) The non-wandering set $\Omega(f)$ is hyperbolic.
\medskip
\par (2) $\Omega(f)$ is the closure of periodic points.
\end{definition}
\medskip
We give the following characterization of Axiom A for PCF maps on $\mathbb{P}^2$. Some interesting examples of Axiom A PCF maps can be found in Ueno \cite{ueno2007dynamics}.
\begin{corollary}
Let $f$ be a PCF map on $\mathbb{P}^2$ of degree $\geq 2$. Then $f$ is Axiom A if and only if $f$ is expanding on $J_2$, and for every critical component cycle $C$, $C\cap J_1$ is a hyperbolic (saddle) set.
\end{corollary}
\begin{proof}
By Theorem 1.4 we know that $\Omega(f)$ is the closure of periodic points. Thus $f$ is Axiom A if and only if $f$ is hyperbolic on $\Omega(f)$. By Theorem 1.4, $\Omega(f)\setminus J_2$ is   the union of super-attracting cycles and $C\cap J_1$, where $C$ is a critical component cycle. Thus $f$ is Axiom A if and only if $f$ is expanding on $J_2$, and for every critical component cycle $C$, $C\cap J_1$ is a hyperbolic (saddle) set (since super-attracting cycle is automatically a unstable dimension $0$ hyperbolic set). 
\end{proof}
\medskip
\section{Further discussion}
In Lemma 2.9 we proved that if  $f$ is a PCF map on $\mathbb{P}^2$, then the set of sporadic  super-saddle cycles  is a finite set. However we do not know any examples of PCF maps on $\mathbb{P}^2$ carrying such sporadic super-saddle fixed points. One may conjecture that such sporadic super-saddle fixed points actually can not exist for PCF maps on $\mathbb{P}^2$. Theorem 3.4 may be seen as an evidence supporting this conjecture.
\medskip
\par It is natural to ask whether our main result Theorem 1.1 can be generalized to higher dimension. Let $f$ be a PCF map on $\mathbb{P}^k$, one may conjecture that in nice enough cases, the  picture might be the following: 
\medskip
\par (1) Let $J_m$ be the Julia set of order $m$. Then $J_m\setminus J_{m+1}$ is contained in the union of the attracting basins of $m$-dimensional periodic submanifolds, where these submanifolds are contained in the critical set.
\medskip
\par (2) There is no Fatou disk  containing $x$ for  every $x\in J_k$.
\medskip 
\par At least two difficulties need to be overcome for proving this conjectural picture. The first is that the classification of points of bounded ramification is not known for higher dimensional PCF maps. In dimension 2 this classification is known by Lemma 2.4. A more serious problem is that, even if the classification of the points of bounded ramification is known, one can possibly prove that for $x\in \mathbb{P}^k\setminus J_k$, $\omega(x)$ is contained in the set of points of unbounded ramification. But it is not possible to show (1), the structure of the Julia sets filtration, by the argument in Theorem 4.5 and 4.8, since  in our argument the points of bounded ramification play an important role.
\medskip 
\bibliographystyle{plain}
\bibliography{PCF}
	
\end{document}